\documentclass[psamsfonts]{amsart}

\usepackage[a4paper,bindingoffset=0.0in,left=.975in,right=.975in,top=1.50in,bottom=1.50in,footskip=.35in]{geometry}

\usepackage{amssymb,amsfonts}
\usepackage[all,arc]{xy}
\usepackage{enumerate}
\usepackage{mathrsfs}

\usepackage{setspace}
\usepackage{tikz-cd}
\tikzset{%
    symbol/.style={%
        draw=none,
        every to/.append style={%
            edge node={node [sloped, allow upside down, auto=false]{$#1$}}}
    }
}

\newtheorem{thm}{Theorem}[section]
\newtheorem{cor}[thm]{Corollary}
\newtheorem{prop}[thm]{Proposition}
\newtheorem{lem}[thm]{Lemma}

\theoremstyle{definition}
\newtheorem{defn}[thm]{Definition}

\theoremstyle{remark}
\newtheorem{rem}[thm]{Remark}

\newcommand{\C}{\Bbb{C}}

\newcommand{\Q}{\Bbb{Q}}

\newcommand{\mS}{\mathcal{S}}

\newcommand{\EE}{\mathcal{E}}

\newcommand{\Mod}{\text{Mod}}

\newcommand{\Id}{\text{Id}}

\newcommand{\Ind}{\text{Ind}}

\newcommand{\ind}{\text{ind}}

\newcommand{\Hom}{\text{Hom}}

\newcommand{\GL}{\text{GL}}

\newcommand{\field}{F}

\makeatletter
\let\c@equation\c@thm
\makeatother
\numberwithin{equation}{section}

\title{Projective smooth representations in natural characteristic}

\author{Amit Ophir and Claus Sorensen}

\date{}

\begin{document}

\begin{abstract}
We investigate under which circumstances there exists nonzero {\it{projective}} smooth $\field[G]$-modules, where $\field$ is a field of characteristic $p$ and $G$ is a locally pro-$p$ group. 
We prove the non-existence of (non-trivial) projective objects for so-called {\it{fair}} groups -- a family including $\bf{G}(\frak{F})$ for a connected reductive group $\bf{G}$ defined over a non-archimedean local field $\frak{F}$. This was proved in \cite{SS24} for finite extensions $\frak{F}/\Bbb{Q}_p$. The argument we present in this note has the benefit of being completely elementary and, perhaps more importantly, adaptable to $\frak{F}=\Bbb{F}_q(\!(t)\!)$. Finally, we elucidate the fairness condition via a criterion in the Chabauty space of $G$. 
\end{abstract}

\maketitle

%\tableofcontents

%-------------------------

\onehalfspacing

\section{Introduction}

Projective objects play a prominent role in the modular representation theory of a finite group $G$. For instance, if $\field$ is a field of characteristic $p$ dividing $|G|$, the Grothendieck group $K_0(\field[G])$ is part of the Cartan-Brauer triangle \cite[p.~46]{Sch13}. Also, in the very definition of the all-important stable module category of $\field[G]$ one kills morphisms which factor through a projective module. 

One side of the $p$-adic Langlands correspondence (or rather its mod $p$ counterpart) involves smooth $\field[G]$-modules where $G$ is now an infinite $p$-adic reductive group, and $\field$ is still a field of characteristic $p$. (A $G$-representation $V$ is \emph{smooth} if the action $G \times V \longrightarrow V$ is continuous for the discrete topology on $V$.)
The smooth $\field[G]$-modules form an abelian category $\Mod_{\field}(G)$ and one could hope to recast modular representation theory in this generality. As was shown in \cite{SS24} this situation is dramatically different from the case of finite groups: There are \emph{no} projective objects in $\Mod_{\field}(G)$ other than $V=\{0\}$.  

We mention in passing that the complex case is dissimilar: If $G$ is a $p$-adic reductive group, and we fix a character $\chi$ of the center, the category
$\Mod_{\C}^\chi(G)$ of smooth $\C[G]$-modules with central character $\chi$ has lots of projective objects. In fact, any irreducible supercuspidal representation is projective (and injective). See \cite{AR04} for more precise results in this direction -- including a converse. 

One of the goals of this note is to extend parts of \cite{SS24} to groups $G=\bf{G}(\frak{F})$ where $\bf{G}$ is a connected reductive group over \emph{any} non-archimedean local field 
$\frak{F}$ -- possibly of positive characteristic. The arguments in \cite{SS24} make heavy use of Poincar\'{e} subgroups, and therefore only apply for finite extensions of $\Bbb{Q}_p$. 
We stress that \cite{SS24} had a different goal (to understand the derived functors of smooth induction) and the non-existence of nonzero projective objects was a byproduct. 
The methods of this note are more elementary and completely avoid cohomology. 

We work in greater generality. We continue to let $\field$ denote a field of characteristic $p$, but we allow $G$ to be any locally pro-$p$ group (by which we mean it admits an open subgroup which is pro-$p$ in the induced topology). If $G$ is discrete, smoothness is automatic, and we obviously have plenty of projective $\field[G]$-modules -- such as $\field[G]$ itself. 
For that reason we will always assume $G$ is non-discrete. 

The key hypothesis is the following: We assume $G$ admits an open subgroup $K$ such that for all open subgroups $H \subset K$ there exists an open subgroup
$H' \subsetneq H$ for which we have \emph{strict} inclusions
\begin{equation}\label{HYP}
K \cap gH'g^{-1} \subsetneq K \cap gHg^{-1} 
\end{equation}
for all $g \in G$. This condition also appeared in \cite{SS24} where it was shown to control the vanishing of the top derived functor 
$R^d\Ind_K^G(\field)$ for a $d$-dimensional $p$-adic Lie group $G$. A pair $(G,K)$ satisfying (\ref{HYP}) is called \emph{fair} in this note. Non-trivial reductive groups $G=\bf{G}(\frak{F})$ as above always admit a subgroup $K$ for which $(G,K)$ is fair. In fact one can take \emph{any} compact open subgroup $K \subset G$, as follows easily from Bruhat-Tits theory. 
We remark that in general (\ref{HYP}) implies $G$ is non-discrete (take $H=\{e\}$). 

With this terminology our main result is the following:

\begin{thm}\label{intromain}
Let $\field$ be a field of characteristic $p>0$. Let $G$ be a locally pro-$p$ group which admits an open subgroup $K$ such that $(G,K)$ is fair, i.e. satisfies (\ref{HYP}). 
Then the category of smooth $\field[G]$-modules $\Mod_{\field}(G)$ has \underline{no} nonzero projective objects. 
\end{thm}

This extends \cite[Thm.~3.1]{CK23} to \emph{locally} pro-$p$ groups, and it partially generalizes \cite{SS24} to local fields $\frak{F}=\Bbb{F}_q(\!(t)\!)$ of characteristic $p$. 

In fact our Theorem \ref{cmain} gives a stronger result than Theorem \ref{intromain}: In Section \ref{cchar} we consider the category of representations with a fixed central character. 
More precisely, we fix a closed central subgroup $C \subset G$, a continuous character $\chi: C \rightarrow \field^\times$, and we show that the category $\Mod_{\field}^\chi(G)$ (of smooth 
$\field[G]$-modules on which $C$ acts via $\chi$) has no nonzero projective objects if $(G/C, KC/C)$ is fair ($\Rightarrow$ $C$ is not open). For the sake of exposition we have emphasized the case where $C$ is trivial here in the introduction. 

Fixing the central character may seem like a nuance, but the categories appearing in the $p$-adic local Langlands program for $\GL_2(\Q_p)$ consist of representations with a fixed central character. More precisely one considers locally admissible smooth $\field[\GL_2(\Q_p)]$-modules with central character $\chi$ (and similarly for more general coefficient rings 
$\mathcal{O}$ instead of $\field$). See \cite{Pas13} for example. 

In Section \ref{exact} we suggest one way out of the no projectives conundrum, which is to endow $\Mod_{\field}(G)$ with a coarser exact structure relative to which there {\it{are}} enough projectives. We also discuss the corresponding stable category, following \cite{Kel96}. This bears a resemblance to the relative homological approach of \cite[Sect.~2, Sect.~5]{DK23}.

In Section \ref{chab} we give a topological criterion for $(G,K)$ being fair, in terms of the Chabauty space $\mS(G)$ of all closed subgroups. We reproduce an argument of Pierre-Emmanuel Caprace proving that $(G,K)$ is fair if and only if the closure of the $G$-conjugacy class of $K$ contains no discrete subgroups.

%-----------------------------------End of introduction----------------------------------------------

\section{Preliminary remarks for profinite groups}\label{prof}

Let $K$ be an infinite profinite group. We let $\Omega$ denote the set of open subgroups of $K$. 

\begin{lem}\label{idx}
The index $[K:U]$ becomes arbitrarily large as $U \in \Omega$ varies. 
\end{lem}

\begin{proof}
Start with any $U \in \Omega$. Since $K$ is infinite we may pick an element $u \in U\backslash \{e\}$. Since $U$ is open there is a $U' \in \Omega$ such that $uU'\subset U$. By choosing $U'$ small enough we can arrange that $e \notin uU'$. Clearly $U' \subsetneq U$, and consequently $[K:U']>[K:U]$. Thus we can make the index arbitrarily large. 
\end{proof}

We fix a field $\field$ and consider the category $\Mod_{\field}(K)$ of smooth $K$-representations on $\field$-vector spaces. Recall that a representation $V$ is smooth if 
every $v \in V$ has an open stabilizer -- in other words $v$ is fixed by some $U \in \Omega$. In particular $\dim_{\field} \field[K]v<\infty$. We let $V^U$ denote the subspace of $U$-fixed vectors in $V$.

\begin{defn}\label{omegav}
For $v \in V$ as above, we let $\Omega_v$ denote the set of $U \in \Omega$ for which 
\begin{itemize}
\item[(a)] $U$ fixes $v$, and 
\item[(b)] $[K:U]>\dim_{\field} \field[K]v$. 
\end{itemize}
\end{defn}  

Note that $\Omega_v \neq \varnothing$ by Lemma \ref{idx}. By Frobenius reciprocity, each vector $v \in V^U$ corresponds to a morphism in $\Mod_{\field}(K)$,
\begin{align*}
\varphi_{U,v}: \ind_U^K(\field) &\longrightarrow V \\
             f &\longmapsto {\sum}_{\kappa \in U \backslash K} f(\kappa) \kappa^{-1}v.
\end{align*}
 Here $\ind_U^K(\field)$ is the space of functions $f: U \backslash K \rightarrow \field$, on which $K$ acts via right translations. This is a finite-dimensional smooth $K$-representation of dimension $[K:U]$. Obviously $\text{im}(\varphi_{U,v})=\field[K]v$, so $\varphi_{U,v}$ is \emph{not} injective when $U \in \Omega_v$. 
 
We consider the sum of all these morphisms,
\begin{align*}
\varphi: S=\joinrel=\bigoplus_{v\in V} \bigoplus_{U \in \Omega_v} \ind_U^K(\field) &\longrightarrow V \\
             (f_{U,v})_{U,v} &\longmapsto {\sum}_{U,v} \varphi_{U,v}(f_{U,v}). 
\end{align*}
Clearly $\varphi$ is surjective since $\varphi_{U,v}(\text{char}_U)=v$ (and any $v\in V$ is fixed by some $U \in \Omega_v$). If $V$ is a projective object of $\Mod_{\field}(K)$
there exists a section $\sigma: V \rightarrow S$ of $\varphi$ in $\Mod_{\field}(K)$. Thus $\varphi \circ \sigma=\text{Id}_V$.

\begin{prop}\label{case}
Suppose $\field$ is a field of characteristic $p>0$.
If $p^\infty$ divides $|K|$ there are no nonzero projective objects in $\Mod_{\field}(K)$, and conversely. 
\end{prop}

\begin{proof}
First assume $p^\infty$ divides the pro-order $|K|$, and pick a Sylow pro-$p$-subgroup $K' \subset K$ (which is infinite by assumption). The restriction functor 
$\Mod_{\field}(K) \rightarrow \Mod_{\field}(K')$ preserves projective objects since $\ind_{K'}^K$ is an exact right adjoint functor (as $K$ is compact). We may 
therefore assume that $K$ is an infinite pro-$p$ group.  

If $V$ is a projective object of $\Mod_{\field}(K)$, we consider a section $\sigma$ of $\varphi$ as above. The section restricts to an embedding 
$\sigma: V^K \hookrightarrow S^K$. As noted earlier, $\ker(\varphi_{U,v})\neq \{0\}$ when $U \in \Omega_v$. Therefore the inclusion
$$
\{0\}\neq \ker(\varphi_{U,v})^K \subset \ind_U^K(\field)^K=\{\text{constants}\}
$$
is an equality. In particular $\varphi_{U,v}$ vanishes on the constant functions, and consequently $\varphi$ vanishes on $S^K$. Since $\varphi \circ \sigma=\text{Id}_V$ we deduce that
$V^K=\{0\}$, which is equivalent to $V=\{0\}$. (See \cite[Lem.~1, p.~111]{AW67} for example.)

For the converse, suppose $p$ has finite exponent in $|K|$. Then there exists a $U \in \Omega$ such that $p \nmid |U|$. A standard averaging argument shows the functor $(-)^U$ is exact on $\Mod_{\field}(K)$. By Frobenius reciprocity this amounts to $\ind_U^K(\field)$ being a projective object in $\Mod_{\field}(K)$.
\end{proof}

This result (Proposition \ref{case}) was proved independently in \cite[Thm.~3.1]{CK23} using a different method. 

\section{The general case}\label{gen}

We now take $G$ to be a \emph{locally} profinite group, by which we mean it has an open subgroup $K$ which is profinite in the induced topology. We assume $G$ is not discrete, i.e. 
any such $K$ is infinite. We choose a $K$ once and for all, and continue to let $\Omega=\{\text{open subgroups of $K$}\}$. 

\begin{defn}\label{fair}
We say the pair $(G,K)$ is \underline{fair} if $\forall H \in \Omega$ there is an $H' \in \Omega$ such that 
$$
K \cap gH'g^{-1} \subsetneq K \cap gHg^{-1}
$$
for all $g \in G$. The group $G$ is fair if $(G,K)$ is fair for some profinite open subgroup $K \subset G$. 
\end{defn}

In what follows $\Ind_K^G$ denotes the full smooth induction functor, i.e. the \emph{right} adjoint to the restriction functor $\Mod_{\field}(G) \rightarrow \Mod_{\field}(K)$. Note that $\Ind_K^G$ is not exact in general; see \cite{SS24}. To fix ideas we adopt the convention that $G$ acts by right translations on induced representations. 

\begin{rem}
The fairness condition (Df. \ref{fair}) also appeared in \cite{SS24}. For a $d$-dimensional $p$-adic Lie group $G$, and $K$ a compact open subgroup, it is shown in \cite{SS24}
that $(G,K)$ is fair if and only if $R^d\Ind_K^G(\field)=0$. (Here $R^i\Ind_K^G$ is the $i^\text{th}$ right derived functor of $\Ind_K^G$.)
\end{rem}

Now $V$ denotes an object of $\Mod_{\field}(G)$. We will often denote its restriction $V|_K$ simply by $V$ when there is no risk of confusion. From Section \ref{prof} we have the morphism 
$\varphi: S \twoheadrightarrow V$ in $\Mod_{\field}(K)$. 

\begin{prop}\label{ind}
The induced morphism $\Ind_K^G(\varphi)$ is surjective if $(G,K)$ is fair (cf. Def. \ref{fair}). 
\end{prop}

\begin{proof}
Start with an arbitrary $F \in \Ind_K^G V$ and pick an $H \in \Omega$ fixing $F$. We simplify the notation by introducing
$v_x:=F(x) \in V^{K \cap xHx^{-1}}$ for all $x \in G$. If $H'$ satisfies the condition in Def. \ref{fair}
we see that
$$
[K: K \cap xH'x^{-1}]>[K: K \cap xHx^{-1}]\geq \dim_{\field} \field[K]v_x. 
$$
Thus $K \cap xH'x^{-1} \in \Omega_{v_x}$, and it makes sense to consider the contribution to $S$ indexed by $v=v_x$ and $U=K \cap xH'x^{-1}$. 
$$
\text{\underline{Claim}: $\forall x \in G$ there is an $f_x \in S^{K \cap xH'x^{-1}}$ such that $\varphi(f_x)=v_x$.}
$$
To see this, consider the morphism $\varphi_{K \cap xH'x^{-1},v_x}$. It maps the characteristic function $\text{char}_{K \cap xH'x^{-1}}$ to $v_x$. We take $f_x$ to be 
$\text{char}_{K \cap xH'x^{-1}}$ viewed as a vector in the summand $\ind_{K \cap xH'x^{-1}}^K(\field)$ of $S$ indexed by $v=v_x$ and $U=K \cap xH'x^{-1}$. This proves the claim. 

Choose a set of representatives $R'$ for $K \backslash G/H'$. This uniquely determines an $A \in (\Ind_K^G S)^{H'}$ such that $A(r')=f_{r'}$ for all $r' \in R'$. We check that 
$\Ind_K^G(\varphi)(A)=F$. Let $g \in G$ be arbitrary, and write $g=\kappa r' h'$ with $\kappa \in K$, $r' \in R'$, and $h' \in H'$. Then
$$
\Ind_K^G(\varphi)(A)(g)=\varphi(A(g))=\varphi(A(\kappa r'))=\kappa \varphi(A(r'))=\kappa \varphi(f_{r'})=\kappa v_{r'}.
$$
On the other hand, since $F$ is fixed by $H' \subset H$, 
$$
F(g)=F(\kappa r')=\kappa F(r')=\kappa v_{r'}.
$$
This shows that indeed $\Ind_K^G(\varphi)(A)=F$, and as $F$ is arbitrary $\Ind_K^G(\varphi)$ is surjective. 
\end{proof}

Adjunction gives us the commutative diagram
$$
\begin{tikzcd}
\Hom_{\Mod_{\field}(K)}(V,S) \arrow[r, "\sim"] \arrow[d, dashed, "\varphi_*"]
& \Hom_{\Mod_{\field}(G)}(V,\Ind_K^G S) \arrow[d, "\Ind_K^G(\varphi)_*"] \\
\Hom_{\Mod_{\field}(K)}(V,V) \arrow[r, "\sim"]
& \Hom_{\Mod_{\field}(G)}(V,\Ind_K^G V).
\end{tikzcd}
$$
Assume $(G,K)$ is fair. If $V$ is a projective object of $\Mod_{\field}(G)$, the morphism $\Ind_K^G(\varphi)_*$ in the diagram is surjective by Proposition \ref{ind}. Hence so is the dashed morphism 
$\varphi_*$. In particular $\varphi$ admits a section $\sigma$ in $\Mod_{\field}(K)$. 

\begin{thm}\label{main}
Let $\field$ be a field of characteristic $p>0$. Suppose $(G,K)$ is fair for some infinite pro-$p$ open subgroup $K \subset G$. Then 
there are no nonzero projective objects in $\Mod_{\field}(G)$.
\end{thm}

\begin{proof}
By the preliminary remarks leading up to the Theorem, $\varphi$ admits a section $\sigma$ in $\Mod_{\field}(K)$. The (second paragraph of the) proof of Proposition \ref{case} now applies verbatim.
\end{proof}

For future reference perhaps it is worth highlighting the following reformulation of Theorem \ref{main}.

\begin{cor}\label{center}
Keep the setup and assumptions from Theorem \ref{main}. Suppose $\chi:\frak{Z}\rightarrow \field^\times$ is a continuous character of a profinite abelian subgroup $\frak{Z}\subset G$
with pro-order prime-to-$p$. Then the subcategory $\Mod_{\field}^\chi(G)$ (=smooth $\field[G]$-modules on which $\frak{Z}$ acts by $\chi$) has no nonzero projective objects.
\end{cor}

\begin{proof}
For $V \in \Mod_{\field}(G)$ we let $V^\chi:=\{v\in V: zv=\chi(z)v, \forall z \in \frak{Z}\}$ denote the $\chi$-eigenspace. The resulting functor $(\cdot)^\chi$ is right adjoint to the inclusion functor 
$\iota: \Mod_{\field}^\chi(G)\rightarrow \Mod_{\field}(G)$. Once we observe $(\cdot)^\chi$ is exact, $\iota$ preserves projectives (and we are done by \ref{main}). 

%Note that $\chi$ is trivial on the Sylow pro-$p$-subgroup of $\frak{Z}$ (since $\field^\times$ has no elements of $p$-power order $>1$). Therefore, to show $(\cdot)^\chi$ is exact we may assume $\frak{Z}$ has pro-order prime-to-$p$. 

The usual averaging argument applies: For $v \in V$, fixed by some small enough open subgroup $U$, consider 
$$
\tilde{v}:=\frac{1}{[\frak{Z}: \frak{Z}\cap U]} \cdot {\sum}_{z\in \frak{Z}/\frak{Z}\cap U} \chi(z^{-1})zv \in V^\chi.
$$
If $\gamma: V \rightarrow V'$ is a morphism in $\Mod_{\field}(G)$, and $v \in V$ is a vector for which $\gamma(v) \in V'^\chi$, then clearly $\gamma(\tilde{v})=\gamma(v)$. Thus, if $\gamma$ is surjective, then so is $\gamma^\chi: V^\chi \rightarrow V'^\chi$.
\end{proof}

%----------------------

\section{Representations with a fixed central character}\label{cchar}

In this section we discuss how to adapt the previous arguments to the category of representations with a fixed central character. 
Our setup is the following: The group $G$ is locally pro-$p$ and we fix an open subgroup $K$ which is pro-$p$ in the induced topology. The field $\field$ has characteristic $p$. We pick a closed central subgroup $C\subset Z(G)$ along with a continuous character $\chi: C \rightarrow \field^\times$ and consider the category $\Mod_{\field}^\chi(G)$ of smooth $\field[G]$-modules on which $C$ acts by $\chi$. (We observe that $\chi$ is automatically trivial on any pro-$p$ subgroup of $C$ such as $C \cap K$.)

We will assume $(G,K)$ is $C$-fair in the following sense (cf. Df. \ref{fair}, which is the case where $C$ is the trivial subgroup):

\begin{defn}\label{cfair}
We say the pair $(G,K)$ is $C$-\underline{fair} if $\forall H \in \Omega$ there is an $H' \in \Omega$ such that 
$$
K \cap gH'Cg^{-1} \subsetneq K \cap gHCg^{-1}
$$
for all $g \in G$. (Here we keep the notation $\Omega:=\{\text{open subgroups of $K$}\}$.) Equivalently, $(G/C,KC/C)$ is fair in the sense of Df. \ref{fair}. (To see this note that every open subgroup of $KC/C$ has the form $HC/C$ for an open subgroup $H \subset K$, and vice versa.)
\end{defn}

This implies that $C$ is {\it{not}} open (by taking $H=C \cap K$ in Df. \ref{cfair}). As a result thereof, Lemma \ref{idx} generalizes:

\begin{lem}\label{cidx}
The index $[K:U(C \cap K)]$ becomes arbitrarily large as $U \in \Omega$ varies. 
\end{lem}

\begin{proof}
Start with any $U \in \Omega$. Pick an element $u \in U\backslash C$ (which is possible as $C$ is not open). 
There is a $U' \in \Omega$ such that $uU'\subset U\backslash C$ (as $C$ is closed). Clearly $U'(C \cap K) \subsetneq U(C \cap K)$ -- otherwise one can write
$u=u'c$ with $u' \in U'$ and $c\in C$ which leads to the contradiction $c \in uU'$.
\end{proof}

Instead of the set $\Omega_v$ from Df. \ref{omegav} we consider the set $\Omega_{C,v}$ of all $U \in \Omega$ fixing $v$ such that
$$
[K:U(C \cap K)]> \dim_{\field} \field[K]v.
$$
By Lemma \ref{cidx} this set $\Omega_{C,v}$ is non-empty. Here $v$ is a vector in an object $V$ of $\Mod_{\field}(K/C \cap K)$. Therefore, for 
$U \in \Omega_{C,v}$, we have a morphism $\varphi_{U,v}: \ind_{U(C \cap K)}^K(\field)\rightarrow V$ in the latter category with nonzero kernel. We again consider their direct sum
$$
\varphi: S={\bigoplus}_{v \in V} {\bigoplus}_{U \in \Omega_{C,v}} \ind_{U(C \cap K)}^K(\field) \longrightarrow V.
$$

The right adjoint of the restriction functor $\Mod_{\field}^\chi(G)\rightarrow \Mod_{\field}(K/C \cap K)$ is given as follows: First we extend $V$ to a representation 
of $KC$ by letting $C$ act via $\chi$. This gives an object $V \boxtimes \chi$ of $\Mod_{\field}^\chi(KC)$ which we induce to a representation $\Ind_{KC}^G(V \boxtimes \chi)$ in
$\Mod_{\field}^\chi(G)$. 

Mimicking Proposition \ref{ind}, we now start with an object $V$ from $\Mod_{\field}^\chi(G)$. We restrict $V$ to $K$ and consider the morphism $\varphi: S \twoheadrightarrow V$ constructed above.

\begin{prop}\label{cind}
The induced morphism $\Ind_{KC}^G(\varphi\boxtimes \chi)$ is surjective if $(G,K)$ is $C$-fair (cf. Df. \ref{cfair}). 
\end{prop}

\begin{proof}
Let $H \in \Omega$ and start with an $F \in \Ind_{KC}^G(V \boxtimes \chi)^H$. Now $v_x:=F(x) \in (V \boxtimes \chi)^{KC \cap xHx^{-1}}$. Choose an $H'$ as in \ref{cfair}. 
To run the proof of Proposition \ref{ind} in this context, it remains to note that
$$
[K: (KC \cap xH'x^{-1})(C \cap K)]>[K: (KC \cap xHx^{-1})(C \cap K)].
$$
If this inequality was an equality we would have 
$$
(KC \cap xH'x^{-1})C=(KC \cap xHx^{-1})C
$$
which contradicts the strict inclusion in \ref{cfair}.
\end{proof}

The rest of the proof of Theorem \ref{main} now extends word for word and gives:

\begin{thm}\label{cmain}
Let $\field$ be a field of characteristic $p$. Suppose $(G,K)$ is $C$-fair for some pro-$p$ open subgroup $K \subset G$. Then 
there are no nonzero projective objects in $\Mod_{\field}^\chi(G)$ for all characters $\chi: C \rightarrow \field^\times$.
\end{thm}

\begin{proof}
In this setup we have the commutative diagram
$$
\begin{tikzcd}
\Hom_{\Mod_{\field}(K/C \cap K)}(V,S) \arrow[r, "\sim"] \arrow[d, dashed, "\varphi_*"]
& \Hom_{\Mod_{\field}^\chi(G)}(V,\Ind_{KC}^G (S\boxtimes \chi)) \arrow[d, "\Ind_{KC}^G(\varphi\boxtimes \chi)_*"] \\
\Hom_{\Mod_{\field}(K/C \cap K)}(V,V) \arrow[r, "\sim"]
& \Hom_{\Mod_{\field}^\chi(G)}(V,\Ind_{KC}^G (V\boxtimes \chi)).
\end{tikzcd}
$$
If $V$ is projective in $\Mod_{\field}^\chi(G)$ we find that $\varphi: S \twoheadrightarrow V$ admits a section $\sigma$. Since $\varphi_{U,v}$ vanishes on the constant functions, 
$\varphi$ must vanish on $S^K$ which contains $\sigma(V^K)$. Thus $V^K=0 \Longrightarrow V=0$.
\end{proof}

%-------------------------

\section{Examples of fair pairs}

A fair group is obviously not discrete (take $H=\{e\}$ in Def. \ref{fair}). In this section we give some basic examples of fair groups. 

\subsection{Groups with a non-discrete center}

Let $G$ be a locally profinite group with non-discrete center $Z(G)$. Then $(G,K)$ is fair for \emph{any} profinite open subgroup $K \subset G$. 

\begin{proof}
Let $H \in \Omega$. Pick a $z \in Z(G) \cap H$, $z\neq e$. There exists an $H' \in \Omega$, contained in $H$, such that 
$$
e \notin z(Z(G) \cap H').
$$
Now, for the sake of contradiction, suppose there is a $g \in G$ for which $K \cap gH'g^{-1}=K \cap gHg^{-1}$. Intersecting both sides with $Z(G)$ yields the equality
$Z(G) \cap H'=Z(G) \cap H$. For instance,
$$
Z(G) \cap (K \cap gHg^{-1})=K \cap g(Z(G) \cap H)g^{-1}=K \cap (Z(G) \cap H)=Z(G) \cap H.
$$
However, the equality $Z(G) \cap H'=Z(G) \cap H$ implies $e \in z(Z(G) \cap H')$. This contradicts the properties of $H'$. 
\end{proof}

\subsection{Reductive groups over local fields}

Let $G=\bf{G}(\frak{F})$ for a connected reductive group $\bf{G}$ defined over a (non-archimedean) local field $\frak{F}$. (We emphasize that we allow $\frak{F}$ to have \emph{positive} characteristic.) Then $(G,K)$ is fair for any compact open subgroup $K$. 

\begin{proof}
It suffices to show $(G,K)$ is fair for a cofinal system of compact open subgroups, so we may assume $K$ is a principal congruence subgroup. As in \cite{SS24}, we pick a special vertex $x_0$ in the Bruhat-Tits building and consider the principal congruence subgroups $K_m$ of the special parahoric subgroup $K_0$. We have implicitly fixed a maximal $\frak{F}$-split
subtorus $\bf{S}$ such that $x_0$ lies in the associated apartment. We let $Z=\bf{Z}(\frak{F})$ denote the $\frak{F}$-points of the centralizer $\bf{Z}$ of $\bf{S}$, and $Z^+$ is the usual contracting monoid (see \cite{SS24} for more details). By the Cartan decomposition $G=K_0Z^+K_0$ it suffices to show that $\forall n$ there is an $n'>n$ such that
$$
K_m \cap zK_{n'}z^{-1} \subsetneq K_m \cap zK_{n}z^{-1} 
$$
for all $z \in Z^+$. This follows immediately from the Iwahori factorization of $K_m \cap zK_{n}z^{-1}$ as in \cite{SS24}.
\end{proof}

In the next subsection we give an alternative proof which works in a more general setup. 

\subsection{Groups with a weak Cartan decomposition}

We say $G$ has a \emph{weak Cartan decomposition} if there is a compact subset $\mathcal{C} \subset G$, and a non-discrete subgroup $S \subset G$, such that
$$
G=\mathcal{C} \cdot Z_G(S) \cdot \mathcal{C}
$$
where $Z_G(S)$ denotes the centralizer of $S$ in $G$. 

\begin{thm}\label{weak_Cartan_implies_fairness}
Suppose that $G$ has a weak Cartan decomposition.
Then $(G,K)$ is fair for any compact open subgroup $K \subset G$. 
\end{thm}

\begin{rem}
Such $G$ are reminiscent of groups with a \emph{Cartan-like decomposition}, cf. \cite[Df.~3.1]{CW23}, which means $G=\mathcal{C}\cdot A\cdot \mathcal{C}$ for a compact open \emph{subgroup} $\mathcal{C} \subset G$ and a set $A$ of representatives for the double cosets
$\mathcal{C} \backslash G/\mathcal{C}$.
\end{rem}

\begin{proof}
We start by recalling the \emph{tube lemma}: Let $X,Y$ be topological spaces, and assume $Y$ is compact. Suppose $\mathcal{V} \subset X\times Y$ is an open subset containing a slice 
$\{x\} \times Y$. Then there exists an open subset $U \subset X$ such that $\mathcal{V}\supset U \times Y \supset \{x\}\times Y$.

For lack of a reference we indicate a proof hereof: Write $\mathcal{V}=\bigcup_{i\in I} U_i \times V_i$ as a union of open boxes. Then 
$Y=\bigcup_{i\in I'}V_i$ where $I'=\{i\in I: x \in U_i\}$, which we refine to a finite subcovering $Y=V_{i_1}\cup \cdots \cup V_{i_N}$. One immediately checks that 
$U=U_{i_1}\cap \cdots \cap U_{i_N}$ satisfies the requirements. 

Consider the continuous map 
\begin{align*}
\xi: S \times \mathcal{C}  & \longrightarrow G \\
 (s,c) & \longmapsto csc^{-1}.
\end{align*}
It maps the slice $\{e\}\times \mathcal{C}$ to $\{e\}$. In particular $\xi^{-1}(K) \supset \{e\}\times \mathcal{C}$, and therefore $\xi^{-1}(K)$ contains a tube $U \times \mathcal{C}$ for some open neighborhood $U \subset S$ of $e$. This means:
\begin{equation}\label{incK}
\{csc^{-1}: s\in U, \: c \in \mathcal{C}\} \subset K.
\end{equation}
Similarly, given an $H$ as in Definition \ref{fair}, the same argument applied to the map $(s,c) \mapsto c^{-1}sc$ yields an open neighborhood 
$V \subset S$ of $e$ such that
\begin{equation}\label{incH}
\{c^{-1}sc: s\in V, \: c \in \mathcal{C}\} \subset H.
\end{equation}
We may assume $V \subset U$. 

Once and for all we pick an element $\sigma \in V-\{e\}$ (which exists since $S$ is non-discrete) and introduce the compact subset
$$
\Sigma:=\{c^{-1}\sigma c: c \in \mathcal{C}\} \subset G.
$$
Since the complement $G-\Sigma$ is an open neighborhood of $e$ there is an open subgroup $H' \subsetneq H$ with $H' \cap \Sigma=\varnothing$. We claim this $H'$ works in \ref{fair}. If not, there is a $g \in G$ for which we get an equality
$$
K \cap gH'g^{-1}=K \cap gHg^{-1}.
$$
Write $g=czc'$ according to the weak Cartan decomposition ($c,c' \in \mathcal{C}$ and $z \in Z_G(S)$). Then:
\begin{itemize}
\item[i.] $c\sigma c^{-1} \in K$ by (\ref{incK});
\item[ii.] $c\sigma c^{-1}=gc'^{-1}z^{-1}\sigma zc'g^{-1}=gc'^{-1}\sigma c'g^{-1} \in gHg^{-1}$ by (\ref{incH}).
\end{itemize}
(In part ii we used the fact that $z$ and $\sigma$ commute.) In summary $c\sigma c^{-1}\in K \cap gHg^{-1}$. By our hypothesis on $g$ this element 
$c\sigma c^{-1}$ must therefore lie in $K \cap gH'g^{-1}$. Consequently
$$
c'^{-1}\sigma c'=g^{-1}cz\sigma z^{-1}c^{-1}g=g^{-1}c\sigma c^{-1}g \in g^{-1}(K \cap gH'g^{-1})g=g^{-1}Kg \cap H' \subset H'.
$$
On the other hand $c'^{-1}\sigma c' \in \Sigma$. This contradicts the assumption that $H' \cap \Sigma=\varnothing$.
\end{proof}

This applies in particular to {\it{covering}} groups of $\bf{G}(\frak{F})$, as discussed in \cite[Sect.~3.1]{FP22} for example. They consider central extensions 
$\widetilde{\bf{G}(\frak{F})}$ of $\bf{G}(\frak{F})$ by a finite abelian group. In \cite[Sect.~5.4]{FP22} the Cartan decomposition of $\bf{G}(\frak{F})$ is lifted to 
a Cartan decomposition of $\widetilde{\bf{G}(\frak{F})}$. See \cite[Thm.~5.3]{FP22} for instance. 

\begin{thm}
Suppose that $G$ has a weak Cartan decomposition $G=\mathcal{C} \cdot Z_G(S) \cdot \mathcal{C}$.
Let $A$ be a closed central subgroup of $G$ such that $A\cap S$ is not open in $S$, and let $\chi:A\rightarrow \field^\times$ be a smooth character.
Then there are no nonzero projectives in $\Mod_{\field}^\chi(G)$.
\end{thm} 
    \begin{proof}
    By Theorem \ref{weak_Cartan_implies_fairness} and Theorem \ref{cmain}, it is enough to show that the weak Cartan decomposition of $G$ induces a weak Cartan decomposition on $G/A$.
    Given a set $B\subset G$, we denote by $\overline{B}\subset G/A$ its image under the quotient map.
    Clearly, $G/A=\overline{\mathcal{C}}\cdot \overline{Z_G(S)}\cdot \overline{\mathcal{C}}$, and $\overline{\mathcal{C}}$ is compact.
    Since $\overline{Z_G(S)}\subset Z_{G/A}(\overline{S})$, we have $G/A=\overline{\mathcal{C}}\cdot Z_{G/A}(\overline{S})\cdot \overline{\mathcal{C}}$.
    By assumption, $A\cap S$ is not open in $S$, hence $\overline{S}$ is not discrete in $G/A$.
    Therefore, $G/A=\overline{\mathcal{C}}\cdot Z_{G/A}(\overline{S})\cdot \overline{\mathcal{C}}$ is a weak Cartan decomposition of $G/A$.
    \end{proof}

%----------------------Exact

\section{Other exact structures}\label{exact}

One solution addressing the lack of projective objects in $\Mod_{\field}(G)$ is to endow this category with other exact structures. The survey \cite{Buh10} serves as our main reference for the basic notions and properties of exact categories. 

Let $\EE_\text{max}$ be the class of all short exact sequences $0 \rightarrow V'\rightarrow V \rightarrow V''\rightarrow 0$ in $\Mod_{\field}(G)$. For a fixed open subgroup $U \subset G$ we consider the class $\EE_U \subset \EE_\text{max}$ of all such sequences which split in $\Mod_{\field}(U)$. The axioms in \cite[Df.~2.1]{Buh10} are easy to verify. This is precisely 
\cite[Exc.~5.3]{Buh10} applied to the restriction functor $\Mod_{\field}(G)\rightarrow \Mod_{\field}(U)$ (with split exact sequences). Thus $(\Mod_{\field}(G),\EE_U)$ is an exact category.

\begin{rem}\label{depu}
If $U' \subset U$ are compact open subgroups of $G$ with index $[U:U']\in \field^\times$ an immediate averaging argument shows that $\EE_{U'}=\EE_{U}$. (If $V \rightarrow V''$ admits a 
$U'$-equivariant section $\sigma$ then $\widetilde{\sigma}(v)=\frac{1}{[U:U']}\sum_{u\in U' \backslash U} u^{-1}\sigma(uv)$ defines a $U$-equivariant section.) 
\end{rem}

An admissible monic is a morphism $\alpha: V' \rightarrow V$ in $\Mod_{\field}(G)$ which admits a $U$-equivariant retraction (a morphism $\rho: V \rightarrow V'$ in 
$\Mod_{\field}(U)$ such that $\rho \circ \alpha=\Id_{V'}$). We use the notation $\rightarrowtail$ for admissible monics.

Similarly, an admissible epic is a morphism $\beta: V \rightarrow V''$ in $\Mod_{\field}(G)$ which admits a $U$-equivariant section (a morphism 
$\sigma: V'' \rightarrow V$ in $\Mod_{\field}(U)$ such that $\beta \circ \sigma=\Id_{V''}$). We use the notation $\twoheadrightarrow$ for admissible epics.

Projective and injective objects of $(\Mod_{\field}(G),\EE_U)$ are defined in \cite[Df.~11.1]{Buh10}. For example, $P$ is projective if every admissible epic
$\beta: V\twoheadrightarrow V''$ induces a surjective map $\beta_*: \Hom_{\Mod_{\field}(G)}(P,V)\rightarrow \Hom_{\Mod_{\field}(G)}(P,V'')$. (See also \cite[Prop.~11.3]{Buh10}.)

\begin{prop}\label{injpr}
Let $U$ be an open subgroup of $G$. Then the following holds:
\begin{itemize}
\item[(a)] The exact category $(\Mod_{\field}(G),\EE_U)$ has enough projectives (see \cite[Df.~11.9]{Buh10}) and enough injectives.
\item[(b)] The projective objects are precisely the direct summands of representations of the form $\ind_U^G(W)$ with $W \in \Mod_{\field}(U)$. The injective objects are the summands of
$\Ind_U^G(W)$ as $W$ varies.
\item[(c)] If $G$ is compact, $(\Mod_{\field}(G),\EE_U)$ is a Frobenius category (see \cite[Sect.~13.4]{Buh10}). 
\end{itemize}
\end{prop}

\begin{proof}
For part (a) let $X$ be an arbitrary object of $\Mod_{\field}(G)$ and consider the counit of adjunction
\begin{align*}
B: \ind_U^G(X|_U) &\longrightarrow X \\
f & \longmapsto {\sum}_{g \in U \backslash G} g^{-1}f(g).
\end{align*}
This is an admissible epic. Indeed $B$ has a $U$-equivariant section $x \mapsto f_x$, where $f_x$ denotes the function in $\ind_U^G(X|_U)$ supported on $U$ and sending the identity to $x$. Also, $\ind_U^G(X|_U)$ is projective in $(\Mod_{\field}(G),\EE_U)$ for the following reason. Any admissible epic $\beta: V\twoheadrightarrow V''$ as above induces a surjective map
$$
\Hom_{\Mod_{\field}(U)}(X|_U,V) \longrightarrow \Hom_{\Mod_{\field}(U)}(X|_U,V'')
$$
since $\sigma \circ (-)$ is a right inverse. By Frobenius reciprocity this amounts to $\beta_*$ being surjective. We conclude that for any $X$ there is an admissible epic $P \twoheadrightarrow X$ with $P$ projective.

To show $(\Mod_{\field}(G),\EE_U)$ also has enough injectives we consider the unit of adjunction
\begin{align*}
A: X &\longrightarrow \Ind_U^G(X|_U) \\
 x &\longmapsto [g \mapsto gx].
\end{align*}
This is an admissible monic. A $U$-equivariant retraction is given by evaluation at the identity. It remains to note that 
$\Ind_U^G(X|_U)$ is injective. So, let $\alpha: V' \rightarrowtail V$ be an admissible monic as above. Again, by Frobenius reciprocity for the full induction, it suffices to observe that 
pulling back via $\alpha$ induces a surjective map
$$
\Hom_{\Mod_{\field}(U)}(V, X|_U)\longrightarrow \Hom_{\Mod_{\field}(U)}(V', X|_U)
$$
 since $(-)\circ \rho$ is a left inverse. Hence any $X$ admits an admissible monic $X \rightarrowtail I$ into an injective $I$. 
 
 In the previous proofs of the projectivity of $\ind_U^G(X|_U)$ and the injectivity of $\Ind_U^G(X|_U)$ there was nothing special about $X|_U$. We can run the exact same arguments for any $U$-representation $W$ instead of $X|_U$. Altogether this proves parts (a) and (b). 
 
 Part (c) follows immediately from (b) since $\ind_U^G(W)=\Ind_U^G(W)$ when $G$ is compact. 
\end{proof}

Any Frobenius category has an associated stable category, which is triangulated. As described in \cite[Sect.~6]{Kel96} this construction can be mimicked in greater generality. 

In our setup the injectively stable category $S_\text{in}(G,U)$ has the same objects as $\Mod_{\field}(G)$ but the morphisms are 
$$
\Hom_{S_\text{in}(G,U)}(V_1,V_2)=\Hom_{\Mod_{\field}(G)}(V_1,V_2)/\mathcal{I}(V_1,V_2)
$$
where $\mathcal{I}(V_1,V_2)$ is the space of morphisms $V_1\rightarrow V_2$ which factor through an injective object (with respect to the exact structure $\EE_U$). 

The suspension functor $T: S_\text{in}(G,U) \longrightarrow S_\text{in}(G,U)$ has the property that there is a short exact sequence
$$
0 \rightarrow X \overset{A}{\rightarrow} \Ind_U^G(X|_U) \rightarrow T(X) \rightarrow 0
$$
in $\EE_U$ for all $X$ (where $A$ is the adjunction map which appeared in the proof of Proposition \ref{injpr}). In other words $T(X)$ is the cokernel of $A$.
By \cite[Thm.~6.2]{Kel96} this gives $S_\text{in}(G,U)$ the structure of a {\it{suspended}} category (see \cite[Sect.~7]{Kel96}). Essentially what this means is it satisfies all the axioms for a triangulated category except that the suspension functor need not be an equivalence.

Similarly, the projectively stable category $S_\text{pr}(G,U)$ is defined by modding out morphisms which factor through a projective object. In this case there is a functor 
 $\Omega: S_\text{pr}(G,U) \longrightarrow S_\text{pr}(G,U)$ such that there is an exact sequence
 $$
0 \rightarrow \Omega(X) \rightarrow \ind_U^G(X|_U) \overset{B}{\rightarrow} X \rightarrow 0
$$
in $\EE_U$ for all $X$ (with $B$ as in the proof of Proposition \ref{injpr}). Thus $\Omega(X)$ is the kernel of $B$.

When $G$ is compact we have a Frobenius category. In this case $S_\text{in}(G,U)=S_\text{pr}(G,U)$ and $T$, $\Omega$ are mutually quasi-inverse equivalences of categories. This gives a 
triangulated category $S(G,U)$.

When $G$ is finite and $U=\{e\}$ the above construction yields the stable module category $S(G)$ which is of pivotal importance in modular representation theory
(when $|G|$ is divisible by the characteristic of $\field$). See \cite{BIK12} for example. 

We are optimistic that $S_\text{in}(G,U)$ and $S_\text{pr}(G,U)$ will likewise play a central role in modular representation theory for non-compact groups, and we hope to explore this in continuation of this paper. 

%-----------------------Chabauty

\section{An interpretation of fairness in the Chabauty space}\label{chab}

The set $\mS(G)=\{\text{closed subgroups of $G$}\}$ carries a natural topology which makes it a compact Hausdorff space; see \cite{Cha50}. The group $G$ acts on  
$\mS(G)$ by conjugation, and for a $K \in \mS(G)$ we will consider the $G$-orbit $\{gKg^{-1}: g \in G\}\subset \mS(G)$ and its closure. 

We owe the following observation, and its proof, entirely to Pierre-Emmanuel Caprace. We are very grateful to him for allowing us to include his argument here. 

\begin{prop}\label{caprace}
The pair $(G,K)$ is fair if and only if
$$
\overline{\{gKg^{-1}: g \in G\}} \subset \{\text{non-discrete closed subgroups of $G$}\}.
$$
\end{prop}

\begin{proof}
First we assume $(G,K)$ is \emph{not} fair. This means there is some $H \in \Omega$ such that for all open subgroups $H' \subset H$ there exists a $g \in G$ for which we have an equality
$$
K \cap gH'g^{-1}=K \cap gHg^{-1}.
$$
We pick a neighborhood basis at the identity $\{H_i\}_{i \in I}$, for some directed set $I$, consisting of open subgroups $H_i \subset H$. For each $i \in I$ we select a $g_i \in G$ with the property that
$$
K \cap g_iH_ig_i^{-1}=K \cap g_iHg_i^{-1}.
$$
Equivalently, $g_i^{-1}Kg_i \cap H \subset H_i$. Consider the net $(g_i^{-1}Kg_i)_{i \in I}$ in $\mS(G)$. Since $\mS(G)$ is compact we can extract a convergent subnet 
$(g_{f(j)}^{-1}Kg_{f(j)})_{j \in J}$ for some reindexing function $f: J \rightarrow I$. Call the limit $\Delta$. We claim $\Delta$ is discrete, which will finish the proof of the {\it{if}} part (by contraposition). 

Intersection with $H$ gives a continuous map (see \cite[Prop.~2.2]{HS14} for instance)
\begin{align*}
\mS(G) & \longrightarrow \mS(H) \\
C & \longmapsto  C \cap H.
\end{align*}
We deduce that $g_{f(j)}^{-1}Kg_{f(j)} \cap H$ converges to $\Delta \cap H$. On the other hand this net converges to $\{e\}$ since $g_i^{-1}Kg_i \cap H \subset H_i$ and $H_i \rightarrow \{e\}$. As $\mS(G)$ is Hausdorff we conclude that $\Delta \cap H=\{e\}$, and in particular $\Delta$ is discrete. 

To prove the \emph{only if} part, assume there is some discrete group $\Delta$ in the closure of $\{g^{-1}Kg: g \in G\}$. We write $\Delta$ as a limit of a net 
$(y_j^{-1}Ky_j)_{j \in J}$ (for some possibly new directed set $J$). Since $\Delta$ is discrete, $\Delta \cap H=\{e\}$ for some open subgroup $H \subset K$. For every 
open subgroup $H' \subsetneq H$ there is a $j$ for which
$$
y_j^{-1}Ky_j \cap H \subset H',
$$
using that $y_j^{-1}Ky_j \cap H \rightarrow \{e\}$. We infer that the inclusions below are equalities:
$$
y_j^{-1}Ky_j \cap H \subset y_j^{-1}Ky_j \cap H' \subset y_j^{-1}Ky_j \cap H.
$$
Conjugation by $y_j$ shows that 
$$
K \cap y_jHy_j^{-1}=K \cap y_jH'y_j^{-1}.
$$
Therefore $(G,K)$ is not fair.
\end{proof}

In the previous proof we used the topology of geometric convergence on $\mS(G)$. This is identical to the Chabauty topology by \cite[Lem.~2, p.~880]{GR06}, for example. 

\begin{rem}
Caprace has informed us that the so-called Neretin groups \emph{fail} to satisfy Proposition \ref{caprace}: For such groups there are choices of $K$ for which 
$\{gKg^{-1}: g \in G\}$ has $\{e\}$ as an accumulation point. 
\end{rem}

%--------------------------------------

\subsection*{Acknowledgments} 

We thank Alexander Lubotzky, Alireza Salehi Golsefidy, and Pierre-Emmanuel Caprace for sharing their knowledge throughout this project.

%--------------------------------------

%-----------------------------------------------------------------------------------------------------------

\bigskip

\noindent {\it{E-mail addresses}}: {\texttt{aophir@ucsd.edu, \: csorensen@ucsd.edu}}

\noindent {\sc{Department of Mathematics, UCSD, La Jolla, CA, USA.}}

\end{document}